\documentclass[a4paper,10pt]{amsart}
\usepackage{amsmath,amssymb,amsthm,amscd}
\usepackage[all]{xy}
\renewcommand{\iff}{if and only if }
\newcommand{\st}{such that }

\DeclareMathOperator{\Hom}{Hom}

\DeclareMathOperator{\End}{End}

\DeclareMathOperator{\Img}{Im}

\DeclareMathOperator{\pd}{dim}

\newcommand{\N}{\mathbb{N}}

\usepackage[usenames]{color}

\theoremstyle{plain}
\newtheorem{thm}{Theorem}[section]
\newtheorem{prop}[thm]{Proposition}
\newtheorem{lem}[thm]{Lemma}
\newtheorem*{conj}{Conjecture A}
\newtheorem{cor}[thm]{Corollary}
\theoremstyle{definition}

\theoremstyle{remark}
\newtheorem*{rem}{Remark}

\begin{document}
\title[Pure semisimplicity conjecture and Artin problem]{Pure semisimplicity conjecture and Artin problem for dimension sequences}

\author{\textsc{Jan \v Saroch}}
\address{Charles University, Faculty of Mathematics and Physics, Department of Algebra \\ 
Sokolovsk\'{a} 83, 186 75 Praha~8, Czech Republic}
\email{saroch@karlin.mff.cuni.cz}


\thanks{Research supported by GA\v CR 17-23112S}

\date{\today}

\begin{abstract} Inspired by a recent paper due to Jos\'e Luis Garc\'ia, we revisit the attempt of Daniel Simson to construct a counterexample to the pure semisimplicity conjecture. Using compactness, we show that the existence of such counterexample would readily follow from the very existence of certain (countable set of) hereditary artinian rings of finite representation type.

The existence of such rings is then proved to be equivalent to the existence of special types of embeddings, which we call \emph{tight}, of division rings into simple artinian rings. Using the tools by Aidan Schofield from 1980s, we can show that such an embedding $F\hookrightarrow M_n(G)$ exists provided that $n<5$. As a byproduct, we obtain a division ring extension $G\subseteq F$ such that the bimodule ${}_GF_F$ has the right dimension sequence $(1,2,2,2,1,4)$.

Finally, we formulate Conjecture~A, which asserts that a particular type of adjunction of an element to a division ring can be made, and demonstrate that its validity would be sufficient to prove the existence of tight embeddings in general, and hence to disprove the pure semisimplicity conjecture.
\end{abstract} 

\maketitle
\vspace{4ex}


A ring $R$ is called \emph{right pure semisimple} if every right $R$-module is a direct sum of indecomposable finitely generated right $R$-modules. The assertion that, over a~right pure semisimple ring $R$, there is only finitely many isomorphism classes of indecomposable right $R$-modules, i.e.\ that $R$ is of finite representation type, is known as \emph{the pure semisimplicity conjecture (pssc)}. Auslander proved the conjecture for Artin algebras in \cite{A}. He also showed that a right pure semisimple ring $R$ is of finite representation type if and only if it is left pure semisimple. From the works \cite{Si1}, \cite{Si2}, \cite{H} of Simson and Herzog from 1980s and 1990s, we know that, to prove the conjecture in its full generality, it suffices to prove it for countable hereditary rings, and also that pssc holds true for all rings with polynomial identity. This has been so far the best general result in the positive direction even though many interesting papers on the topic appeared in the meantime; notably by D\~ung--Garc\'ia and Angeleri H\"ugel (\cite{DG1}, \cite{DG2}, \cite{DG3}, \cite{An}, \cite{AH}).

Daniel Simson also made considerable effort in the direction towards disproving the conjecture. He proposed a plausible way how to construct many counterexamples to the pssc, cf.\ \cite{Si}. The problem is that to construct these counterexamples one has to construct division ring extensions with certain special properties, and so far no one knows how to do it. Generalizing Schofield's tools from \cite[Chapter~13]{S0} seemed very promising for some time, however, to the best of author's knowledge, there has been virtually no further development in this direction.

Recently published inconspicuous papers \cite{G2} and \cite{G} by Garc\'ia opened a way to revisit Simson's approach and hopefully also draw an otherwise declining attention of the algebraic community back to this fascinating long-standing open problem.

In this article, we do not aspire to give a thorough overview of the tools and machinery that has been developed around the pssc in the last almost five decades. The author does not even consider himself an expert in the field, and so his endeavour cannot be anything but modest: he would like to give a little nudge to the topic he finds captivating, and perhaps convince other mathematicians that it is worth a shot.

\smallskip

The paper is organized as follows. In the first section, we capitalize on Garc\'ia's characterization (found in \cite[Section~6]{G}) of the squared small counterexample to the pssc suggested by Simson, and show that its existence follows, by the compactness of the first-order logic, from the mere existence of a particular countable set of hereditary artinian rings of finite representation type; see Theorem~\ref{t:finsat}. In Section~\ref{sec:Equiv}, we prove that there is a correspondence between these hereditary artinian rings of finite representation type and special types of embeddings of division rings into simple artinian rings which we call \emph{tight}.

In Section~\ref{sec:conjA}, we formulate Conjecture~A, which asserts that a particular type of adjunction of an element to a division ring can be made, and demonstrate that its validity would be sufficient to prove the existence of tight embeddings in general, and hence to disprove the pure semisimplicity conjecture. We also show that some cases of Conjecture~A hold true, and as a consequence, we provide an example of a division ring extension $G\subseteq F$ where ${}_GF_F$ has the right dimension sequence $(1,2,2,2,1,4)$; see Section~\ref{sec:pssc} for unexplained terminology.

The last section is devoted to showing that, if Conjecture~A holds true, then the existentially closed models of certain first-order axiomatization of $1$-tight embeddings are actually tight.

\section{Preliminaries, squared small counterexample and some logic}
\label{sec:pssc}

In the whole paper, a \emph{ring} means an associative unital ring. If $M$ is a $(G,F)$-bimodule where $G, F$ are division rings, we denote by $M^r$ the right dual of $M$, i.e. the $(F,G)$-bimodule $\Hom_F(M,F)$. The \emph{right dimension sequence} of $M$ is defined as the sequence $(\pd M_F, \pd (M^r)_G, \pd (M^{rr})_F,\dots)$. In case of a periodic right dimension sequence, we often state just a finite period. If the length of this period is not clear, we also explicitly specify this length.


If $G\subseteq F$ is a division ring extension, we denote by $R_F$ the triangular matrix ring $\begin{pmatrix} G & F \\ 0 & F \end{pmatrix}$. Following \cite[Section~6]{G}, the ring $R_F$ is a \emph{squared small counterexample} to the pssc if the right dimension sequence of the $(G,F)$-bimodule $F$ is $(1,2,2,2,\dots)$ and $\pd {}_GF = \infty$. Note that $R_F$ is really a counterexample to the pssc: it is right pure semisimple and not of finite representation type. We refer to \cite{Si} for exposition and more details.

We follow the right-handed version of the notation from \cite[Section~6]{G} which we first fit into logical context. Let $\mathcal L$ be the first-order language of (unital) rings extended by the unary relation symbol $G$ and countably many constant symbols $f_0,f_1,f_2,\dots$. As usual, in what follows, we often write $(\exists x\in G)\,\varphi(x,\bar y)$ instead of $(\exists x)\,G(x)\,\&\,\varphi(x,\bar y)$, and $(\forall x\in G)\,\varphi(x,\bar y)$ instead of $(\forall x)\,G(x) \rightarrow \varphi(x,\bar y)$.
For each positive integer $k$, let $D_k(x)$, $S_k(x)$ respectively, denote the formula $$\bigwedge _{j = 0}^{k-1} (\exists a_0,\dots, a_k\in G)\ f_jx = \sum_{i = 0}^k a_if_i,$$

$$\bigwedge _{j = 0}^{k} (\exists a_0,\dots, a_k\in G)\ f_jx = \sum_{i = 0}^k a_if_i\hbox{ respectively.}$$

\medskip

Let $T$ be the $\mathcal L$-theory with the following set of axioms:
\begin{enumerate}
	\item $1 = f_0$;
	\item the axioms of a division ring;
	\item the axioms saying that $G$ is a division subring;
	\item $(\forall x)(\exists a,b\in G)\ x = a + f_1b$;
	\item for each $k\in\mathbb N$ the axiom $(\forall a_0,\dots, a_k\in G)\, \sum_{i=0}^{k} a_if_i = 0 \rightarrow \bigwedge _{i =0}^k a_i = 0$;
	\item for each $k\in\mathbb N$ the axiom $$(\exists b)(D_k(b)\;\&\; (\forall x)(D_k(x) \rightarrow (\exists a_0,a_1)(x= a_0+ba_1 \;\&\; S_k(a_0)\;\&\; S_k(a_1)))).$$
\end{enumerate}

Let us explain the axioms assuming for a moment that the $\mathcal L$-theory $T$ is satisfiable. In a model $F$ of the theory $T$, let $G = S_0 = \{f\in F\mid G(f)\mbox{ holds in }F\}$ be the division subring defined by the formula $G(x)$. Further, for $k\in\mathbb N_0$, let $L_k$ be the left $G$-subspace of $F$ spanned by $\{f_0,f_1,\dots,f_k\}$. Also for each $k\in\N$, we denote by $S_k = \{f\in F \mid L_k f \subseteq L_k\}$ the subring of $F$ defined by the formula $S_k(x)$. The multiplication by an arbitrary nonzero $f\in S_k$ from the right provides an endomorphism of the left $G$-module $L_k$. Since $F$ is a division ring, this endomorphism has got trivial kernel, whence $L_k f = L_k$, and it immediately follows that $S_k$ is, in fact, a division subring of $F$. Similarly, $D_k = \{f\in F \mid L_{k-1}f \subseteq L_k\}$ forms the $(S_{k-1},S_k)$-bimodule defined by the formula $D_k(x)$.

The axioms in $T$ state that $G\subseteq F$ is a division ring extension (the axioms $(2)$ and $(3)$) such that the left $G$-dimension of $F$ is $\infty$ (since $\{f_i \mid i\in\N_0\}$ is a left independent subset of $F$ over~$G$ by $(5)$), the right $G$-dimension of $F$ is $2$ (by $(4)$) and $\pd {D_k}_{S_k} \leq 2$ for all $k\in\N$ (by the axiom $(6)$).

\begin{lem}\label{l:regular} Let $k\in \N$ and $F$ be a model of $T$. Then for each $f\in F$ there exist $a_k\in L_{k-1}$ and $b_k\in S_{k-1}$ such that $f = a_k + f_kb_k$. Moreover, dim ${D_k}_{S_k} = 2$.
\end{lem}

\begin{proof} By induction for $k$. If $k = 1$, we have $S_{k-1} = G = L_{k-1}$ and the result follows from the axiom $(4)$.

Let $k\geq 1$ and assume that the claim holds for all $1\leq i\leq k$. From the induction hypothesis, we have $f_{k+1} = a_k + f_k b_k$ for some $a_k\in L_{k-1}$ and $b_k\in S_{k-1}$. Hence $b_k\in D_k\setminus S_k$ (due to axiom $(5)$), and so dim ${D_k}_{S_k} = 2$ using the axiom $(6)$.

Now let $f\in F$ be arbitrary. From the induction hypothesis, we have $a\in L_{k-1}$ and $b\in S_{k-1}$ such that $f = a + f_k b$. Since $b\in D_k$, we get $b = s + b_k t$ for a (unique) pair $(s,t)\in S_k^2$. Putting everything together, we have
$$f = a+f_k(s+b_kt) = a + f_ks+ (f_{k+1} -a_k)t = a+f_ks-a_kt+f_{k+1}t$$ and it remains to put $a_{k+1} = a+f_ks-a_kt$ and $b_{k+1} = t$.
\end{proof}

Recall from \cite{G} that an extension of division rings $G\subseteq F$ is called \emph{right regular} if for any two left $G$-subspaces $H_1, H_2$ of $F$ with finite and equal dimensions there exists $a\in F$ \st $H_1a = H_2$. The following proposition is basically \cite[Proposition~6.3]{G} with an additionally shown redundace of the assumption of right regularity.

\begin{prop}\label{p:models} $F\models T$ \iff $R_F$ is a squared small counterexample to the pure semisimplicity conjecture.
\end{prop}

\begin{proof} If $R_F$ is a squared small counterexample to the pssc then $F\models T$ (for a~suitable choice of $1 = f_0, f_1,\dots$) by \cite[Lemma~5.2 and Proposition 6.3]{G}. Conversely assume that $F\models T$. From the axioms, we get dim ${}_G F = \infty$, dim $F_G = 2$ and dim ${D_k}_{S_k} \leq 2$. In fact, by Lemma~\ref{l:regular}, dim ${D_k}_{S_k} = 2$. We are going to prove that the extension $G\subseteq F$ is right regular.

It is enough to show, by induction on $k\in\N_0$, that, for a left $G$-subspace $H$ of $F$ with dimension $k+1$, there exists $a\in F$ such that $L_k a = H$. This is trivial for $k = 0$, so let $k>0$. Using the induction hypothesis, we can assume that $H = L_{k-1} + Gf$ for $f\not\in L_{k-1}$. From Lemma~\ref{l:regular}, we have $f = a_k + f_k b_k$ for $a_k\in L_{k-1}$ and $b_k \neq 0$, $b_k\in S_{k-1}$. It follows that $H = L_{k-1} + Gf_k b_k$ and $L_k b_k = H$.

The other implication in \cite[Proposition 6.3]{G} concludes the proof.
\end{proof}

For what follows, it is convenient to record a particular case of results in \cite[Theorems~2.8 and 3.9]{Si}.

\begin{lem} \label{l:finrep} Let $G\subseteq F$ be a division ring extension. Assume that the right dimension sequence of the $(G,F)$-bimodule $F$ begins with \[d^{\vee} = (1,\underbrace{2,\dots, 2}_{(n-1)\times}, 1)\mbox{ where }n\in\N.\] Then the ring $R_F$ is of finite representation type with exactly $n+2$ indecomposable modules up to isomorphism. Moreover, the right dimension sequence of~$F$ is (periodic with the period) $d = (1,\underbrace{2,\dots, 2}_{(n-1)\times}, 1,n)$ and $\pd {}_GF = n$.
\end{lem}

\begin{proof} Notice that $d^{\vee}$ is a simple restriction of the dimension-sequence $d$, cf.\ \cite[\S 2.3 and \S 3.2]{Si}. The rest follows from the left-right dual versions of \cite[Theorems~2.8 and 3.9]{Si} used for the $(F,G)$-bimodule $F\cong ({}_GF_F)^r$ and the fact that $\pd M_F = \pd {}_FM^r$ provided that $\pd M_F$ is finite (resp.\ with $F$ replaced by $G$).
\end{proof}

Recall that each finitely generated right $R_F$-module $M$ can be represented by an $F$-linear map $G^t\otimes_G F\to F^s$ where $t,s\in\mathbb N_0$. The pair $(t,s)$ is called \emph{the dimension vector of} $M$ and it is an important invariant of the module $M$.

In the lemma below, we explicate a result implicitly contained in \cite[Sections~5 and 6]{G}. Unlike there, we deal with $R_F$ of finite representation type here, however the reasoning is the same which is why we give just a sketch of the proof.

\begin{lem} \label{l:Garcia} Let $G\subseteq F$ be a division ring extension and $\pd {}_G F = n$ where $n\geq 2$. Let $\{f_0,\dots, f_{n-1}\}$ be a left basis of $F$ over $G$ where $f_0 = 1$. Assume that, for $1\leq k < n$, the $L_k$, $D_k$ and $S_k$ are defined as before. In addition, we put $D_0 = F$. The following conditions are equivalent:
\begin{enumerate}
	\item The $(G,F)$-bimodule $F$ has the right dimension sequence $(1,2,2,\dots, 2,1,n)$ of length $n+2$.
	\item $\pd {D_k}_{S_k} = 2$ whenever $0\leq k< n-1$.
\end{enumerate}
\end{lem}

\begin{proof}[Proof sketch] $(1)\Longrightarrow (2)$. According to \cite[Proposition 2.1]{G}, we can list all the indecomposable right $R_F$-modules up to isomorphism in a sequence $X_0, X_1,\dots, X_n$, $X_{n+1}$. We know that there is $n+2$ such modules by Lemma~\ref{l:finrep}. These modules have the respective dimension vectors $(1,0), (2,1),\dots, (n, n-1)$, $(1,1)$, $(0,1)$ by the left-right dual version of \cite[Proposition~2.11 and pg.\ 240]{G2}; notice that the occurrence of the vector $(1,1)$ stems from the second 1 in the sequence $(1,2,2,\dots, 2,1,n)$. In particular, for $1\leq k<n$, we have $S_k\cong \End_{R_F}(X_k)$ and $D_k\cong \Hom_{R_F}(X_k, X_{k-1})$ by \cite[Proposition~5.1]{G}. Now, $(2)$ follows by \cite[Proposition~2.1(vi)]{G}.

$(2)\Longrightarrow (1)$. As in Proposition~\ref{p:models}, we observe that the division ring extension is right regular; in particular, a finitely generated indecomposable right $R_F$-module is determined up to isomorphism by its dimension vector. The assumption on the left and right dimension of $F$ over $G$ implies that the right dimension sequence of the left dual $\Hom_G({}_GF,G)$ of ${}_GF_F$ begins with $n,1,2,\dots$. (We use $(2)$ for $\pd F_G = \pd {D_0}_{S_0} = 2$.) The modules $M_0, M_1$ with the respective dimension vectors $(1,0)$ and $(2,1)$ are the two indecomposable injective modules; moreover, $M_0\oplus M_1$ is a rigid tilting module according to \cite[Definition 2.4]{G2}. As before, we have $S_1\cong \End_{R_F}(M_1)$ and $D_1\cong \Hom_{R_F}(M_1, M_0)$ by \cite[Proposition~5.1]{G}.

Repeated use of \cite[Lemma 2.7 and Proposition 2.8]{G2} and our assumption $(2)$, gives us the indecomposable modules $M_2, M_3,\dots, M_{n-1},M_n$ with the respective dimension vectors $(3,2), (4,3), \dots, (n,n-1), (1,1)$. Notice that the vector $(1,1)$ occur since $D_{n-1} = S_{n-1} = F$. In turn, \cite[Lemma 2.7]{G2} together with Lemma~\ref{l:finrep} gives us the rest of the right dimension sequence of ${}_GF_F$.
\end{proof}

The final result of this section relates the existence of a squared small counterexample to an instance of Artin problem for dimension sequences, cf.\ \cite[Definition~2.3]{Si}. To the author's best knowledge, it has not been pointed out anywhere in the literature that a mere existence of particular (countable set of) hereditary artinian rings of finite representation type already disproves the pssc. One can certainly learn a lesson here: we should have the finite-representation-type building blocks all investigated and well understood before attempting to construct something fancy of infinite representation type.

\begin{thm}\label{t:finsat} There is a squared small counterexample to the pure semisimplicity conjecture provided that, for infinitely many integers $n\geq 2$, there exists a division ring extension $G\subseteq F$ such that the $(G,F)$-bimodule $F$ has the right dimension sequence $(1,2,2,\dots,2,2,1,n)$ of length $n+2$.
\end{thm}

\begin{proof} We use Proposition~\ref{p:models}. By the compactness theorem from the first-order logic, the theory $T$ is satisfiable \iff $T$ is finitely satisfiable. However, this easily follows from our assumption: for each finite $\Delta\subseteq T$, there is an integer $n$ large enough such that the axioms with $k\geq n$ from schemata $(5)$ and $(6)$ do not occur in $\Delta$, and a division ring extension $G\subseteq F$ such that the $(G,F)$-bimodule $F$ has the right dimension sequence $(1,2,2,\dots, 2,1,n)$ of length $n+2$ exists. In particular, dim $F_G = 2$, dim ${}_G F = n$ and dim ${D_k}_{S_k} \leq 2$ for each $1\leq k<n$ by Lemma~\ref{l:Garcia} (notice that $F = S_{n-1} = D_{n-1}$ trivially holds). Consequently, $F\models \Delta$.
\end{proof}

If $G\subseteq F$ is a division ring extension where ${}_GF_F$ has the right dimension sequence $(1,2,2,\dots,2,2,1,n)$ of length $n+2$, then the hereditary artinian ring $R_F$ is of (finite representation) Coxeter type $I_2(n+2)$. Unfortunately, there is no known construction of such hereditary artinian rings, all the more so such division ring extensions, for $n\geq 5$. The case $n = 3$ is covered in \cite[Section~3]{S} (and also \cite[Chapter~13]{S0}), and we deal with the case $n = 4$ in Section~\ref{sec:conjA}.

\section{Tight embeddings}
\label{sec:Equiv}

Let $F, G$ be division rings and $\varphi:F \to M_n(G)$ a ring embedding where $n\in\N$. Given positive integers $a,b$ such that $a+b = n+1$, we say that $\varphi$ is \emph{$a$-tight} if the following holds: for an arbitrary choice of an $a\times b$ matrix $g$ with elements from $G$, there exists a (necessarily unique\footnote{If there were two such distinct $f,f^\prime\in F$, then $\varphi(f-f^\prime)$ would not be a regular matrix (having the upper right $a\times b$ block full of zeros), however $\varphi((f-f^\prime)^{-1})$ would have to be its inverse.}) $f\in F$ such that $\varphi(f)$ has $g$ as its upper right corner block. If $\varphi$ is $a$-tight for each $a\in\{1,2,\dots, n\}$, we call it \emph{tight}.

It is easy to see that, in the setting $n = 1$, $\varphi$ is tight if and only if it is an isomorphism of division rings. The author was unable to find any reference in the literature concerning the existence of these embeddings for $n>1$. However, as we shall see, this question is tightly related to the existence of the desired division ring extension $G\subseteq F$.

\begin{thm} \label{t:tight} Let $G, F$ be division rings and $n\in\N, n > 1$. The following conditions are equivalent:
\begin{enumerate}
	\item There exists a ring extension $G\subseteq F$ such that the right dimension sequence of the $(G,F)$-bimodule $F$ has length $n+2$ and equals $(1,2,2,\dots, 2,1,n)$.
	\item There exists a tight embedding of $F$ into $M_n(G)$.
\end{enumerate}
\end{thm}

\begin{proof} $(1)\Longrightarrow (2)$. By our assumption and Lemma~\ref{l:finrep}, $\pd {}_G F = n$. Let $B = \{f_0,\dots, f_{n-1}\}$ be a~left basis of $F$ over $G$ where $f_0 = 1$. Using this basis, we can define the embedding of $F$ into $M_n(G)$ via the action of $F$ on ${}_G F$ from the right. In this embedding, each element $a\in F$ corresponds to the matrix $a = (a_{ij})\in M_n(G)$ such that $f_{i-1}a = \sum _j a_{ij}f_{j-1}$ for all $1\leq i,j\leq n$. As in Section~\ref{sec:pssc}, we use, for each $0\leq k< n$, the notation $L_k$ for the left $G$-subspace of $F$ spanned by $\{f_0,\dots, f_k\}$. Further, we set $S_k = \{a\in F \mid L_k a \subseteq L_k\}$, $D_0 = F$, and $D_k = \{a\in F \mid L_{k-1} a \subseteq L_k\}$ for $k\geq 1$. Notice that $S_k$ is a division subring of $F$ whilst $D_k$ is a right $S_k$-module for each $0\leq k<n$. We also have the trivial equalities $D_{n-1} = S_{n-1} = L_{n-1} = F$ and $S_0 = L_0 = G$.

Notice how the matrices in $D_k$ and $S_k$ look like: the subring $S_k$ comprises of all matrices in $F$ having the $(k+1) \times (n-k-1)$ upper right block full of zeros. Similarly, the right $S_k$-module $D_k$ comprises of all matrices in $F$ having the $k \times (n-k-1)$ upper right block full of zeros.

We know from Lemma~\ref{l:Garcia} that our assumption $(1)$ translates to $\pd {D_k}_{S_k} = 2$ for each $0\leq k< n-1$. We will prove that the embedding $F\subseteq M_n(G)$ defined in the previous two paragraphs is tight. This is done by induction on $a$, where $a,b$ are positive integers satisfying $a+b = n+1$. The case $a = 1$ is trivial by the definition of the embedding $F\subseteq M_n(G)$: indeed, the first row of a matrix from $F$ can be arbitrary.

Let us assume that $a\geq 1$ and the claim holds for $a$. To prove it for $a+1$, we are going to use that $\pd {D_{a-1}}_{S_{a-1}} = 2$. Fix the basis $\{1,e\}$ of $D_{a-1}$ over $S_{a-1}$ where $e$ is the unique matrix in $F$ whose first $a-1$ rows end with $n-a+1 = b$ zeros and whose last $b$ entries in the $a$th row are $0,1,0,\dots, 0$. Such an $e$ exists by the inductive hypothesis; using the description of matrices in $D_{a-1}$ and $S_{a-1}$, we see that $e\in D_{a-1}\setminus S_{a-1}$.

If $s\in S_{a-1}$ is arbitrary, then the last $b-1$ entries in the $a$th row of the matrix $es$ coincide with the last $b-1$ entries in the $a+1$th row of $s$. By the inductive hypothesis, the last $b-1$ entries (even the last $b$ entries) in the $a$th row of a matrix from $D_{a-1}$ can be arbitrary elements of $G$. From the fact that $\{1,e\}$ is a basis of $D_{a-1}$ over $S_{a-1}$, we infer that the last $b-1$ entries in the $a+1$th row of a matrix from $S_{a-1}$ can be arbitrary as well. We conclude that, given any $(a+1)\times (b-1)$ matrix $g$ with entries from $G$ there is, by the inductive hypothesis, an $f\in F$ such that $g$ coincides with the upper right corner block of $f$ bar the $a+1$th row. At the same time, we have an $s\in S_{a-1}$ whose last $b-1$ entries in the $a+1$th row coincide with the last row of $g$ whilst the last $b-1$ entries of the first $a$ rows of~$s$ are zero. Hence $f + s$ is the (unique) matrix from $F$ \st $g$ coincides with its upper right corner.

\smallskip

$(2)\Longrightarrow (1)$. The $1$-tightness gives us the embedding of $G$ into $F$ where an element $g\in G$ goes to the unique matrix in $F$ with the first row equal to $(g,0,0,\dots ,0)$. Further, we denote by $f_i, 0\leq i<n$, the matrices in $F$ with the first row equal to $(0,\dots, 0,1,0,\dots, 0)$ where the $1$ is on the $i+1$th position. It is easy to check that, for this choice of $f_0,\dots, f_{n-1}$, the given tight embedding $F\subseteq M_n(G)$ coincides with the one constructed for the division ring extension $G\subseteq F$ in the very first paragraph of the proof. In particular, we see that $\pd {}_G F = n$.

For each $0\leq k<n$, we define the $L_k$, $D_k$ and $S_k$ as before. We are going to show that $\pd {D_k}_{S_k} = 2$ for all $0\leq k<n-1$. Fixing the $k$, we denote by $e$ the (unique) matrix in $F$ whose first $k$ rows end with $n-k$ zeros and whose last $n-k$ entries in the $k+1$th row are $0,1,0,\dots, 0$. We claim that $\{1,e\}$ is a basis of $D_k$ over $S_k$.

From the description of matrices in $D_k$ and $S_k$, we see that $e\in D_k\setminus S_k$. Put $a = k+2$. Given a matrix $d\in D_k$, we use the $a$-tightness to pick $s\in S_k$ whose last $n-k-1 = b$ entries in the $a$th row are the same as the last $b$ entries of the $a-1$th row of $d$. Then the matrix $es\in D_k$ has the same last $b$ entries in the $a-1$th row as $d$. Consequently, $d = es + (d-es)$ where, clearly, $d-es\in S_k$.

We have thus proved that $\pd {D_k}_{S_k} = 2$ for $0\leq k< n-1$. Also, trivially, $D_{n-1} = S_{n-1}$. Using Lemma~\ref{l:Garcia}, we conclude that the $(G,F)$-bimodule $F$ has the right dimension sequence $(1,2,2,\dots, 2,1,n)$ of length $n+2$.
\end{proof}

\begin{rem} Notice that the first paragraphs of the proof of the respective two implications give a bijection between the division ring extensions $G\subseteq F$ with a given left $G$-basis $\{1 = f_0,\dots, f_{n-1}\}$ of $F$ and $1$-tight embeddings $F\to M_n(G)$.
\end{rem}

Unfortunately, the question of existence of tight embeddings for a given $n\in \N$ seems to be a rather hard problem in itself. The construction in Section~\ref{sec:conjA} shows that we can take care of the cases $a,b\leq 2$ which, in turn, implies the existence of tight embeddings for $n<5$. For the proof of the existence of a squared small counterexample to the pssc, it would be enough to show the existence of tight embeddings for infinitely many positive integers $n$. With respect to the compactness result in the first section, it would be even enough to have, for any positive integer $k$, an integer $n\geq k$ large enough such that there is an embedding $F\subseteq M_n(G)$ which is $a$-tight for each $a\in\{1,\dots,k\}$.

\section{Conjecture A and partial results}
\label{sec:conjA}

Of course, the mere translation in the previous section cannot turn the hard problem into an easy one. We have to present a change of perspective which could potentially allow us to make some headway. We propose the following conjecture which essentially speaks about a particular way how to adjoin an element to a~division ring embedded into a simple artinian ring.

\begin{conj}\label{c:conjA} Let $n\in\N, n>1$ and $\varphi:F\to M_n(G)$ be a ring embedding where $G$ and $F$ are division rings. Assume that $a,b$ are positive integers such that $a + b = n+1$. Let $A$ be an $a\times b$ matrix with entries from $G$. Then there exist division ring extensions $G\subseteq G^\prime$ and $F\subseteq F^\prime$ and a ring embedding $\varphi^\prime:F^\prime\to M_n(G^\prime)$ such that $\varphi^\prime\restriction F = \varphi$ and $\Img(\varphi^\prime)$ contains a matrix whose upper right $a\times b$ block is~$A$.
\end{conj}

Writing that ``Conjecture~A holds true'', we mean that the above statement is true for all $n>1$ and every pair $a,b\in\N$ such that $a+b = n+1$.
First, let us observe that Conjecture~A is sufficient to solve our problem.

\begin{prop} \label{p:conjAsuf} Assume that Conjecture~A holds true. Then there exists a tight embedding $\varphi:F\to M_n(G)$ for any given $n\in\mathbb N$.
\end{prop}

\begin{proof} Let us assume that $n>1$ since the case $n = 1$ trivially holds without even assuming Conjecture A. Let us start with an arbitrary embedding $\varphi_{-1}:F_{-1}\to M_n(G_{-1})$ where $F_{-1}, G_{-1}$ are division rings. We constuct $\varphi:F\to M_n(G)$ as the union of countable increasing chain $(\varphi_m:F_m\to M_n(G_m)\mid m\in\{-1,0,1,2,\dots\})$ of embeddings where, for each $m < \omega$, we ensure that $\varphi_m$ is $(m\!\!\mod n)+ 1$-tight.

So let $m<\omega$ be fixed and $\varphi_{m-1}$ already constructed. Set $a = (m\!\!\mod n)+ 1$ and $b= n+1-a$. Let $\mathcal M$ denote the set of all $a\times b$ matrices with entries from $G_{m-1}$. Iteratively using Conjecture A, taking unions in limit steps, we extend $\varphi_{m-1}$ to $\varphi_{m-1}^+:F_{m-1}^+\to M_n(G_{m-1}^+)$ with the property that each $A\in\mathcal M$ occurs as the upper right corner block of a matrix in $\Img(\varphi_{m-1}^+)$. Let $\mathcal M^+$ be the set of all $a\times b$ matrices with entries from $G_{m-1}^+$. We repeat the process with $\varphi_{m-1}^+$ and $\mathcal M^+$ to obtain $\varphi_{m-1}^{++}$, $\mathcal M^{++}$, and so on. Finally, we define $\varphi_m$ as the union of the chain $\varphi_{m-1}\subseteq\varphi_{m-1}^+\subseteq\varphi_{m-1}^{++}\subseteq\dotsb$. It is straightforward to check that $\varphi_m$ is $a$-tight as well as checking that $\varphi = \bigcup_{m<\omega} \varphi_m$ is actually tight.
\end{proof}

\begin{rem} \label{r:submod} In fact, we have shown a little bit more. Assuming Conjecture~A, we actually extended an arbitrary ring homomorphism $\varphi_{-1}:F_{-1}\to M_n(G_{-1})$, where $F_{-1}$ and $G_{-1}$ are division rings, into a tight embedding $\varphi$. Notice that, if the starting homomorphism $\varphi_{-1}$ is already $a$-tight for some $a \in \{1,\dots,n\}$, then we necessarily have $\Img(\varphi_{-1}) = M_n(G_{-1})\cap \Img(\varphi)$. Indeed, if we have a matrix $A$ in the intersection on the right-hand side, it shares the same $a\times (n+1-a)$ upper right corner block with a matrix from $\Img(\varphi_{-1})$ which means that these two matrices actually coincide (their difference is not a regular matrix). The other inclusion is trivial.
\end{rem}

Next, we show how to apply tools from \cite[Chapter~13]{S0} to show that Conjecture~A actually holds true for any $n$ and $a\in\{1,2,n-1,n\}$. For notational details, we refer to \cite{S0} or \cite{S}.

\begin{prop}\label{p:speccases} Conjecture A holds for the setting $n>1$ and $a\in\{1,2,n-1,n\}$.
\end{prop}

\begin{proof} \underline{Case $a = 1$}. This is a special case of \cite[Theorem~9]{S} or \cite[Theorem~13.13]{S0}. For the reader's convenience, we present the part that is actually needed here.

Consider the canonical embedding $F\hookrightarrow M_n(F)$. Then the simple artinian coproduct, \cite[Page~81]{S0}, of $M_n(F)$ and $M_n(G)$ amalgamating $F$ (which is embedded into $M_n(G)$ by $\varphi$) has a form $M_n(F^\prime)$ where $F^\prime$ is a division ring, cf. \cite[Theorem~5.6]{S0}. We get an embedding $\eta:M_n(F)\hookrightarrow M_n(F^\prime)$. We can assume, without loss of generality, that this embedding is actually an inclusion by choosing $\{\eta(f_{ij})\}_{i,j = 1}^n$ as the canonical matrix units (see \cite[Chapter~1.7]{C}) where $f_{ij}\in M_n(F)$ denotes, for each $i,j\in\{1,\dots, n\}$, the matrix with the only nonzero entry $1$ in the position $(i,j)$. Thus we get a division ring extension $F\subseteq F^\prime$. We also have the following commutative diagram consisting of ring embeddings.

$$\begin{CD}
	M_n(F) @>{\subseteq}>> M_n(F^\prime)	\\
	@AAA @A{\psi}AA \\
	F  @>{\varphi}>> M_n(G)
\end{CD}$$

Let $g_{11}\in M_n(F^\prime)$ denote the matrix with a single $1$ in the upper left corner and zero everywhere else. Put $M = M_n(F^\prime)g_{11}$. Then $M$ is naturally an $(M_n(G),F^\prime)$-bimodule. Note that the left action of $M_n(G)$ is induced by $\psi$. We see that $\dim M_{F^\prime} = n$. For each $i = 1,\dots, n$, let $e_{ii}\in M_n(G)$ denote the matrix with the only nonzero element $1$ in the position $(i,i)$. A nontrivial amount of work, \cite[Page~210]{S0}, 
shows that $\{\psi(e_{ii})g_{11}\mid i = 1,\dots, n\}$ forms a basis of $M_{F^\prime}$. Consequently, $(M,g_{11})$ is a pointed cyclic bimodule generated by $g_{11}$, i.e.\ $M= M_n(G)g_{11}F^\prime$.

To define $G^\prime$, we consider the universal localisation of the simple artinian coproduct of $M_n(G)$ and $F^\prime$ with $(M,g_{11})$ amalgamated, cf.\ \cite[Chapter~13]{S0}. This universal localisation is again a simple artinian ring $M_n(G^\prime)$. It also gives us a~commutative diagram consisting of ring embeddings as below.

$$\begin{CD}
	M_n(G) @>{\subseteq}>> M_n(G^\prime)	\\
	@A{\varphi}AA @A{\varphi^\prime}AA \\
	F  @>{\subseteq}>> F^\prime
\end{CD}$$

In $M_n(G^\prime)$, we consider its sub-$(M_n(G),F^\prime)$-bimodule $M_n(G)F^\prime$. By the amalgam construction, there is a bimodule isomorphism of $M$ and $M_n(G)F^\prime$ sending $g_{11}$ to $1$ (see the last paragraph on \cite[pg.\ 199]{S0}\footnote{We can use \cite[Theorem~13.2]{S0} since the multiplication by $\left(\begin{smallmatrix} 0 & x \\ 0 & 0\end{smallmatrix}\right)$, where $x = g_{11}$ in our case, is a full map by the remark immediately preceding the statement of \cite[Theorem~13.2]{S0}.}). It follows that $\{e_{ii}\mid i = 1,\dots, n\}$ forms a basis of $M_n(G)F^\prime$ as a right $F^\prime$-space. In particular, for each $g\in M_n(G)$, there exists an $f\in F^\prime$ such that $e_{11}g = e_{11}f$. Otherwise said, we obtained a universal solution for Conjecture A with $a = 1$.

\smallskip

\underline{Case $a = n$}. The proof is the same as above. We just have to swap sides and let $e_{nn}$ play the role of $e_{11}$.

\smallskip

\underline{Case $a = 2$}. We show how to extend $\varphi$ to a $2$-tight embedding $\varphi^\prime$. Since the case $a = 1$ is done, we can without loss of generality start with $\varphi$ which is $1$-tight (following the proof of Proposition~\ref{p:conjAsuf} caring only about 1-tightness). In particular, there is the embedding $\nu:G\hookrightarrow F$ which sends an element $g\in G$ to the only $f\in F$ such that the first row of the matrix $\varphi(f)$ is $(g, 0, 0, \dots, 0)$. As in the first paragraph of the proof of Theorem~\ref{t:tight} $(2)\Longrightarrow (1)$, we define the elements $1 = f_0,f_1,\dots,f_{n-1}\in F$ and see that they form a left basis of $F$ over $G$.

We use \cite[Theorem~1]{S} or \cite[Theorem~13.12]{S0} to find a division ring embedding $\nu^\prime:G^\prime \hookrightarrow F^\prime$ extending $\nu$ such that the left basis of $F^\prime$ over $G^\prime$ is $\{1,f_1,\dots, f_{n-1}\}$ whilst the right basis is $\{1,f_1\}$. The left basis gives us naturally\footnote{See also the remark following Theorem~\ref{t:tight}.} the extension of $\varphi$ to the $1$-tight embedding $\varphi^\prime: F^\prime\to M_n(G^\prime)$. We claim that it is also $2$-tight.

Recall that the first row of the matrix $\varphi^\prime (f_1)$ is $(0,1,0,0,\dots,0)$. Since the first row of a matrix $\varphi^\prime(f)$, $f\in F^\prime$, can be arbitrary and there exist $c,d\in G^\prime$ such that $f = \nu^\prime(c)+f_1\nu^\prime(d)$, we infer that the matrices in $\Img(\varphi^\prime\nu^\prime)$ can have arbitrary entries in the second row bar the first one, i.e.\ given $g_2,g_3,\dots, g_n\in G^\prime$ there exists (a unique) matrix in $\Img(\varphi^\prime\nu^\prime)$ with the second row equal to $(x,g_2,g_3,\dots,g_n)$ for some $x\in G^\prime$. As a consequence, given $\begin{pmatrix} g_2 & g_3 & \dots & g_n \\ h_2 & h_3 & \dots & h_n \end{pmatrix}$ with entries in $G^\prime$, we see that it is the upper right corner block of the matrix $\varphi^\prime(f)$ for $f\in F^\prime$ such that $f = \nu^\prime(c)+f_1\nu^\prime(d)$ where the second row of $\varphi^\prime\nu^\prime(d)$ is $(x,g_2,\dots, g_n)$ whilst the second row of $\varphi^\prime\nu^\prime(c)$ has the form ``$(y,h_2,\dots,h_n)$ minus the second row of $\varphi^\prime(f_1\nu^\prime(d))$'' for suitable $x,y\in G^\prime$.

\smallskip

\underline{Case $a= n-1$} follows by the same reasoning, starting with an $n$-tight $\varphi$. 	
\end{proof}

Proposition~\ref{p:speccases} gives us the existence of tight embeddings for $n<5$. As a~special case, using Theorem~\ref{t:tight}, we obtain

\begin{cor} \label{c:122214} There exists a division ring extension $G\subseteq F$ where ${}_GF_F$ has the right dimension sequence $(1,2,2,2,1,4)$.
\end{cor}

In particular, $R_F$ corresponding to the extension from Corollary~\ref{c:122214} is an example of a~hereditary artinian ring of non-crystallographic Coxeter type $I_2(6)$ (for explanation, see \cite[Page~393]{R}) as opposed to the crystallographic Coxeter type~$G_2$ represented, for instance, by the triangular matrix algebra $\begin{pmatrix} \mathbb Q & \mathbb Q[\sqrt[3]{2}] \\ 0 & \mathbb Q[\sqrt[3]{2}] \end{pmatrix}$ which has also precisely $6$ indecomposable modules up to isomorphism. The existence of tight embeddings for $n\geq 5$ remains open; see also \cite[Problem~2.10]{Si}.

\begin{rem}\label{r:conjA} The proof of Proposition~\ref{p:speccases} above relies heavily on the ingenious tools developed by Aidan Schofield in 1980s. At that time, he did not succeed in generalizing them further, so that they would be applicable for other cases too (the first one on hand is $n = 5$ and $a = 3$). It seems like a good idea to search for an as elementary as possible (although tedious, perhaps) proof of Conjecture A for $a = 1,2$. It might give us a much needed insight into the character of data we have to control if we want to establish a construction which would succeed also for $a>2$.

Viewed from this perspective, despite still being pretty hard, the translated problem does not look as frightening and impenetrable as it looked when it asked us to control the right dimension of $F$, $F^r$, $F^{rr}$, etc.
\end{rem}

\section{Existentially closed $1$-tight embeddings}
\label{sec:ECem}

It might seem that, in trying to construct tight embeddings, we aim for something very specific, almost pathological. Assuming Conjecture~A, we demonstrate in the sequel that this is not entirely true.

We start by axiomatizing tight embeddings in the similar manner as we did with division ring extensions in the beginning of Section~\ref{sec:pssc}. For that matter, we fix an integer $n>1$ and consider the following theory $T_n$ in the language of (unital) rings extended by a unary relation symbol $F$ and constant symbols $e_{ij}$ where $i,j=1,\dots, n$. The axioms of $T_n$ are:
\begin{enumerate}
	\item the axioms of rings;
	\item the axioms saying that $e_{ij}$ is the full set of matrix units, i.e.\ $\sum_{i=1}^n e_{ii} = 1$, $e_{ij}e_{kl} = e_{il}$ if $j = k$, and $e_{ij}e_{kl} = 0$ otherwise.
	\item the axioms saying that a nonzero element centralized by all $e_{ij}$, $i,j=1,\dots, n$, has an inverse which is also centralized by all the elements $e_{ij}$;
	\item the axioms saying that $F$ defines a division subring;
  \item for each $k\in\{1,\dots,n\}$ the $k$-tightness axiom $$(\forall x)(\exists y\in F)\,\left(\sum_{i=1}^k e_{ii}\right)(x-y)\left(\sum_{i=k}^n e_{ii}\right) = 0.$$
\end{enumerate}

The models of $T_n$, if there are any, are precisely the tight embeddings $F\subseteq M_n(G)$, or more precisely: each model of $T_n$ has the form $M_n(G)$ where $G$ is a~division ring and a full set of $G$-centralizing matrix units in $M_n(G)$ is fixed (the first three axioms); moreover, we have a specified division subring $F\subseteq M_n(G)$ (axiom $(4)$) such that this ring embedding is tight (axiom $(5)$). Of course, the notion of tightness is relative to the fixed full set of matrix units. For more information on matrix units see \cite[Chapter~1.7]{C}.

Consider now the subtheory $T_n^1$ of $T_n$ where we drop all the axioms from $(5)$ with the sole exception of $k = 1$. This time, we know that $T_n^1$ is satisfiable since there are many embeddings $F\subseteq M_n(G)$ which are $1$-tight (as seen in the proof of Theorem~\ref{t:tight}, each division ring extension $G\subseteq F$ with $\dim {}_GF = n$ naturally defines one).

Recall that a model $M$ of a theory $T$ is \emph{existentially closed} if $M$ satisfies an existential sentence with parameters from $M$ provided that this sentence is satisfied in an extension $M^\prime$ of $M$ which is also a model of $T$. Existentially closed models play in general model theory similar role as, say, algebraically closed fields in the theory of fields. For more details and some illustration of the concept, we refer to \cite[Chapter~6.5]{C}. We are interested in existentially closed models of $T_n^1$. Their very existence follows from the fact that $T_n^1$ is an inductive theory (cf.\ \cite[Lemma~3.5.7]{KC}), i.e.\ axiomatized by formulas of the type $(\forall \bar x)(\exists \bar y)\,\varphi(\bar x,\bar y)$ where $\varphi$ is quantifier-free; hence the union of any ascending chain of models of $T_n^1$ is again a model of $T_n^1$. However, if we assume Conjecture~A in addition, we get the following

\begin{prop}\label{p:EC} Assume that Conjecture~A holds true. Then, for any given $n>1$, each existentially closed model of $T_n^1$ is a tight embedding $F\subseteq M_n(G)$.
\end{prop}

\begin{proof} Let $\psi:F\subseteq M_n(G)$ be an existentially closed model of $T_n^1$. Using the remark following Proposition~\ref{p:conjAsuf}, we can extend $\psi$ to a tight embedding $\varphi:F^\prime \subseteq M_n(G^\prime)$ where $F^\prime\cap M_n(G) = F$. Consequently $\psi$ is actually a submodel of $\varphi$. Fix a~$k\in\{2,\dots, n\}$ and consider the axiom from $(5)$ expressing $k$-tightness. Taking any $x\in M_n(G)$ as a parameter, we see that there exists a desired $y\in F^\prime$ by the $k$-tightness of $\varphi$. However, since $\psi$ is existentially closed and understanding $x$ as a~parameter makes the $k$-tightness axiom into an existential sentence, we infer that a~desired $y$ has to exist already in $F$. Thus $\psi$ is $k$-tight because $x$ was arbitrary.
\end{proof}

So assuming Conjecture~A, tight embeddings are abundant in a sense. Furthermore, since Conjecture~A is not needed for the existence of existentially closed models of $T_n^1$, it might be tempting to try to prove directly that existentially closed models of $T_n^1$ are tight. Or that they are not tight and Conjecture~A thus does not hold. The author, however, does not consider this approach to be very promising. After all, the point of existentially closed models is that they bring into existence something that has a potential to exist in a larger model. It is highly unlikely that we gain some intrinsic information about existentially closed models of $T_n^1$ that would help us (dis)prove Conjecture~A. Usually it is the other way around: knowing something about possible model extensions, we gain information about existentially closed models.

\bigskip

\noindent\textbf{Acknowledgement.} I would like to thank the anonymous referees for careful reading and apt suggestions which helped to improve the quality of the paper.


\end{document}